\documentclass[a4,11pt, leqno]{article}
\usepackage{amsmath,amsfonts,amsthm,latexsym,amssymb,mathrsfs}
\usepackage{bbm}

\def\A{{\mathcal A}}
\def\B{{\mathcal B}}
\def\J{{\mathcal J}}
\def\1{\mathbbmss{1}}
\def\n{\mathbb{N}}

\def\L{\mathcal{L}}
\def\X{\mathcal{X}}
\def\K{\mathcal{K}}
\def\C{\mathcal{C}}
\def\H{\mathcal{H}}
\def\T{\mathcal{T}}

\newtheorem{df}{Definition}[section]
\newtheorem{thm}[df]{Theorem}
\newtheorem{pro}[df]{Proposition}
\newtheorem{cor}[df]{Corollary}

\newtheorem{rema}[df] {Remark}

\def\sfstp{{\hskip-1em}{\bf.}{\hskip1em}}

\def\enddemo{\qed \endtrivlist}
\expandafter\let\csname enddemo*\endcsname=\enddemo

\def\qedsymbol{\ifmmode\bgroup\else$\bgroup\aftergroup$\fi
  \vcenter{\hrule\hbox{\vrule
height.6em\kern.6em\vrule}\hrule}\egroup}
\def\qed{\ifmmode\else\unskip\nobreak\fi\quad\qedsymbol}

\begin{document}
\title
{ \bf Isolated spectral points and Koliha-Drazin\\ invertible elements
in quotient Banach\\ algebras and homomorphism ranges\/}

\author {\normalsize Enrico   Boasso }

\date{ }


\maketitle \thispagestyle{empty} 


\setlength{\baselineskip}{12pt}

\begin{abstract}\noindent  In this article poles, isolated spectral points, group, Drazin and 
Koliha-Drazin invertible elements in the context of quotient Banach algebras or in ranges of (not necessarily continuous) homomorphism
between complex unital Banach algebras will be characterized using Fredholm and Riesz Banach algebra elements.
Calkin algebras on Banach and Hilbert spaces will be also considered.\par\vskip.5truecm
\noindent\it Keywords: \rm Isolated spectral point, pole, Koliha-Drazin inverse, Drazin inverse, Banach algebra,
Calkin algebra.
\end{abstract}


\section {\sfstp Introduction}\setcounter{df}{0}

In the  article \cite{Bel}, given a separable infinite dimensional Hilbert space $\H$ and $T\in \L (H)$, it was characterized when zero is a pole of the resolvent of $\pi (T)\in \C(\H)$ 
in terms of the operator $T\in \L (\H)$, where $\C(\H)$ is the Calkin algebra and $\pi\colon \L (\H)\to \C (\H)$ is the quotient map.
In fact, zero is a pole of the resolvent of $\pi (T)\in \C(\H)$ if and only if there are operators $A$, $B$ and $K\in \L(\H)$ such that $A$ is nilpotent,
$B$ is Fredholm, $K$ is compact and $T=A\oplus B+K$ (see \cite[Theorem 2.2]{Bel}). It is worth noticing that according to \cite[Proposition 1.5]{Lu} or \cite[Theorem 12]{Bo},
given a complex unital Banach algebra $\A$, the set of the poles of the resolvent of $a\in \A$ coincide with $\{\lambda \in \hbox{ iso }\sigma (a)\colon \hbox{ such that } a-\lambda 1
\hbox{ is Drazin invertible}\}$, where iso $\sigma (a)$ denotes the set of isolated points of the spectrum $\sigma (a)$ and $1\in \A$ denotes the 
identity of $\A$. Furthermore, according to \cite[Proposition 1.5]{Lu}, 
iso $\sigma (a)=\{\lambda \in \sigma (a)\colon \hbox{ such that } a-\lambda 1
\hbox{ is Koliha-Drazin invertible}\}$.\par

\indent On the other hand, recall that according to Atkinson's theorem, necessary and sufficient for a bounded linear operator on a Banach space
to be  Fredholm is that its coset in the  Calkin algebra is invertible.  Motivated for this result, an abstract Fredholm theory in 
arbitrary complex unital Banach algebras with respect to an inessential ideal was developed, see for example  \cite{A, BMSW} and the
bibliography in these monographs.
In addition,  a Fredholm theory relative to Banach algebra
homomorphisms was introduced in \cite{H1}. Many authors have developed this theory studying for example Fredholm, Weyl,
Browder and Riesz Banach algebra element relative to homomorphisms, see for example \cite{H1, BMSW, H3, H2, H4, MR1, GR, MR, D, H5, KMSD, Bak, ZH, MMH,
ZDH1, ZDH2}.\par

\indent The main objective of this article is to characterize isolated spectral points and Koiha-Drazin invertible elements 
in quotient Banach algebras or in ranges of (not necessarily continuous) homomorphism between Banach algebras. In fact, in section 3, after having recalled some peliminary definitions and facts in section 2,
given two complex unital Banach algebras $\A$ and $\B$, an algebra homomorphism $\T\colon \A\to \B$,
$a\in\A$ and $b=\T (a)$, poles and isolated spectral points of $b\in\B$ will be characterized in terms of Fredholm and Riesz elements  relative to $\T$.
In particular, Drazin and Koliha-Drazin invertible elements in the range of $\T$ will be characterized (see Theorem \ref{thm42} and Corollary \ref{cor43}).  
It is worth noticing that the problem of lifting idempotents will be central in the proof of the aforementioned results. This
problem has been considered by many authors and in general can not be solved, see for example \cite{O, Choi, Bar, Ba, R, D, Had, Kon}.
However, when the lifting is possible, for example when $\T\colon\A\to \B$ is surjective and has the Riesz property (see section 2), $\B=\A/\J$  ($\J\subset \A$
is the radical or a closed inessential ideal and $\T$ is the quotient map $\pi \colon\A\to \A/\J$),  $\A$ is a von Neumann or a real rank zero $C^*$-algebra,
and $\B=\A/\J$ ($\J$ a closed two sided ideal and  $\T=\pi$), or $\A$ is a unital $C^*$-algebra, $\B=\A/\J$ and $\J$ is a closed two sided ideal that is an $\hbox{\rm AW}^*$ algebra  or an AF $C^*$-algebra
(see Remark \ref{rem44}), then poles, isolated spectral points, group, Drazin and Koliha-Drazin invertible elements in $\B$
will be fully characterized using Fredholm and Riesz elements relative to $\T$ or $\pi$.   

\indent In addition, in the Banach space operator context, the following characterization will be proved (compare with \cite{Bel}). Let $\X$ be a Banach space, $T\in \L(\X)$ and $\pi(T)\in\C(\X)$. Then,
$\lambda\in $ iso $\sigma (\pi (T))$ (respectively $\lambda$ is a pole of $\pi (T)$) if and ony if  there exists an idempotent
$P\in \L(\X)$, $P\notin\K (\X)$, such that if $\X_1=P(\X)$  and $\X_2=(I-P) (\X)$, then there are $T_1\in \L(\X_1)$ a Riesz operator (respectively 
a power compact operator),
$T_2\in \L (\X_2)$ a Fredholm operator and $K\in \K(\X)$ such that $T-\lambda=T_1\oplus T_2+K$, where $\K(\X)$ is the 
closed ideal of compact operators. Naturally, the case $\lambda=0$ provides a characterization of
Koliha-Drazin and Drazin invertible elements in the Calkin algebra $\C(\X)$, respectively. In
the frame of arbitrary Hilbert space operators and for isolated points of the spectrum (respectively poles) of Calkin algebra elements,
it will be proved that the operator $T_1$ can be chosen as quasi-nipotent(respectively nilpotent). This last result will be also formulated in
$C^*$-algebras. Finally, as an application of the main results, Drazin and Koliha-Drazin invertible Calkin algebra elements will
be described in terms of the cosets of operator classes. \par

\section {\sfstp Preliminary definitions and facts}\setcounter{df}{0}

From now on $\A$ will denote a complex unital Banach algebra with identity $1$. Let 
$\A^{-1}$  denote the set of all invertible elements in $\A$ and $\A^\bullet$ the set of all 
idempotents in $\A$. Given $a\in \A$, $\sigma (a)$,
$\rho (a)=\mathbb{C}\setminus \sigma (a)$  and iso $\sigma (a)$ will stand for the spectrum, the resolvent set and  the set of isolated spectal points of $a\in \A$,
respectively. Recall that if $K\subseteq \mathbb{C}$, then acc $K$ is the set of limit points of $K$ and
iso $K=K\setminus$ acc $K$. \par

As in the case of   
the spectrum of a Banach space operator, the resolvent function of 
$a\in \A$, $R(\cdot, a)\colon\rho(a)\to \A$, is holomorphic and iso $ \sigma(a)$ is the set of
isolated singularities of $R(\cdot ,  a)$. Furthermore, as in the case of an operator,
see \cite[p. 305]{T}, if $\lambda_0\in$ iso $\sigma(a)$, then it is possible to
consider the Laurent expansion of  $R(\cdot ,  a)$ in terms of $(\lambda-\lambda_0)$.
In fact,
$$
R(\lambda ,  a)=\sum_{n\ge 0}a_n(\lambda-\lambda_0)^n +\sum_{n\ge 1}b_n(\lambda-\lambda_0)^{-n},
$$
where $a_n$ and $b_n$ belong to $\A$ and they are obtained in a standard way using the
functional calculus. In addition, this representation is valid when $0<\mid \lambda-\lambda_0\mid<\delta$, for 
any $\delta$ such that $\sigma(a)\setminus \{\lambda_0\}$ lies outside the circle $\mid\lambda-\lambda_0\mid=\delta$. 
What is more important, the discussion of \cite[pp. 305 and 306]{T}
can be repeated for elements in a unital Banach algebra. Consequently, $\lambda_0$
will be called \it a pole of order $p$ of   $R(\cdot ,  a)$, \rm if there exists $p\ge 1$ such that
$b_p\neq 0$ and $b_m=0$, for all $m\ge p+1$. The set of poles of $a\in \A$
will be denoted by $\Pi(a)\subseteq$ iso $\sigma (a)$.
Define  $I(a)=$ iso $\sigma(a)\setminus \Pi (a)$.\par

\indent Let $\X$ be a Banach space and $\L(\X)$ denote  the algebra of all bounded and linear maps
defined on and with values in $\X$. If $T\in \L(\X)$, then $N(T)$ and $R(T)$ will stand for the null space and the
range of $T$ respectively. Note that $I\in \L(\X)$ will stand for the identity map. Recall that the \it descent \rm and the \it ascent \rm of $T\in \L(\X)$ are
$d(T) =\hbox{ inf}\{ n\ge 0\colon  R(T^n)=R(T^{n+1})\}$ and
$ a(T)=\hbox{ inf}\{ n\ge 0\colon  N(T^n)=N(T^{n+1})\}$
respectively, where if some of the above sets is empty, its infimum is then defined as $\infty$, see  
for example \cite{K}. Now well, according to \cite[Lemma 3.4.2]{CPY}, $\Pi(T)=\{\lambda\in\sigma (a)\colon
a(T-\lambda)\hbox{ and } d(T-\lambda) \hbox{ are finite}\}$. What is more, if $\A$ is a unital Banach algebra and $a\in \A$, 
then, according to  \cite[Theorem 11]{Bo}, $\Pi (a)=\Pi (L_a)=\Pi (R_a)$,
where $L_a\in \L (\A)$ and $R_a\in \L(\A)$ are the operators defined by left and right multiplication respectively,
i.e., given $x\in \A$, $L_a(x)=ax$ and $R_a(x)=xa$.\par

Recall that
$a\in \A$ is said to be \it Drazin invertible, \rm if there exists a necessarily unique $b\in \A$
and some $m\in\n$ such that
$$ab=ba, \hskip.3truecm bab=b, \hskip.3truecm a^mba=a^m.$$

\noindent If the Drazin inverse of $a$ exists, then it will be denoted by $a^d$. In addition,
the  \it index of $a$,\rm  which will be denoted by  ind $(a)$, is the least non-negative integer $m$ for which the
above equations hold. When ind $(a)=1$,
$a$ will be said to be \it group invertible\rm, and in this case its Drazin inverse will be refered as the group inverse
of $a$. The set of Drazin invertible elements of $\A$ will be denoted
by $\A^D$; see \cite{Dr, K, RS}.\par

In addition, in \cite{BS} the \it Drazin spectrum \rm of $a\in \A$ was introduced, 
i.e., $\sigma_{D}(a)=\{\lambda\in \mathbb{C}\colon a-\lambda\hbox{ is not Drazin invertible}\}$.  
Recall that according to \cite[Theorem 12]{Bo}, see also \cite[Proposition 1.5]{Lu},
$\sigma (a)=\sigma_{D} (a)\cup \Pi (a)$,
$\sigma_{D} (a)\cap \Pi (a)=\emptyset$ and $\sigma_{D} (a)=$ acc $\sigma (a)\cup I(a)$.
In particular,  if $ \rho_{\mathcal{D}}(a)=\mathbb{C}\setminus \sigma_{D}(a)$,
then $\Pi (a)=\sigma (a)\cap  \rho_{\mathcal{D}}(a)$, equivalently,
$a\in \A$ is Drazin invertible but not
invertible if and only if $0\in\Pi (a)$; see also \cite[Lemma 3.4.2]{CPY} and \cite[Theorem 4]{K}.
Concerning the Drazin spectrum, see \cite{BS, Lu, Bo}.

An element $a\in \A$ is said to be \it generalized Drazin \rm  or \it Koliha-Drazin invertible\rm,
if there exists $b\in A$ such that
$$ab=ba,\hskip.3truecm bab=b, \hskip.3truecm aba=a+w, \hbox{ ($w\in\A^{qnil}$)},$$

\noindent where $\A^{qnil}$ is the set of all quasinilpotent elements of $\A$. The Koliha-Drazin inverse is unique, if it exists,
and it will be denoted by $a^D$; see \cite{Ko2,Lu, R2, R3}.
Note that if $a\in \A$ is Drazin invertible, then $a$ is generalized Drazin invertible. In fact, 
if $k=$ ind $(a)$, then $(aba-a)^k=0$. In addition, the Koliha-Drazin spectrum was introduced and studied in
 \cite{Lu}, i.e., $\sigma_{KD}(a)=\{\lambda\in \mathbb{C}\colon a-\lambda\hbox{ is not Koliha-Drazin invertible}\}$.
According to \cite[Proposition 1.5]{Lu} (see also \cite[Theorem 4.2]{Ko2}), $\sigma_{KD}(a)=$ acc $\sigma (a)$. In particular,
necessary and sufficient for $a\in \A$ to be Koliha-Drazin invertible but not invertible is that
$0\in$ iso $\sigma (A)$. What is more, $a\in \A$ is generalized Drazin invertible but not
Drazin invertible if and only if $a\in I(\A)$. The set of Koliha-Drazin invertible 
elements of $\A$ will be denoted by $\A^{KD}$.

In the following Remark characterizations of the elements in $\Pi (a)$ and $I(a)$, $a\in \A$, will be recalled.
 \par

\begin{rema}\label{rem11}\rm Let $A$ be a unital Banach algebra and consider $a\in A$. Recall that
if $0\neq p=p^2\in A$, then $pAp$ is a unital Banach algebra with unit $p$. \par
\noindent (i) Let $\lambda\in \mathbb{C}$. Then, $\lambda\in$ iso $\sigma(a)$ if and only if there is $0\neq p^2=p\in A$
 such that $ap=pa$, $(a-\lambda)p$ is quasi-nilpotent  and $(a-\lambda+p)\in A^{-1}$.\par
\noindent (ii) Necessary and sufficient for $\lambda\in \mathbb{C}$ to belong to $\Pi (a)$ is that
there exists $0\neq p^2=p\in A$ such that $ap=pa$, $(a-\lambda)p$ is nilpotent and $(a-\lambda+p)\in A^{-1}$.
Note in particular that $((a-\lambda)p)^l=0$ if and only if $l\ge k=$ ind $(a-\lambda)$.\par
\noindent (iii) The number $\lambda\in \mathbb{C}$ belongs to $I(a)$ if and only if
there is $0\neq p^2=p\in A$ such that $ap=pa$, $(a-\lambda)p$ is quasi-nilpotent but not nilpotent and $(a-\lambda+p)\in A^{-1}$.
\par
\noindent Statements (i)-(ii) are well known, see for example \cite[Proposition 1(a)]{RS}, \cite[Theorem 4.2]{Ko2} and \cite[Theorem 1.2]{Ko}. Statement (iii) is a direct consequence of 
statements (i) and (ii). In addition,  $p\in \A^\bullet$ is the \it spectral idempotent of $a$ corresponding to $\lambda$. \rm
To learn more results on isolated spectral points, see for example \cite{M,S}.\end{rema}

\indent Let $\X$ be an infinite dimensional complex Banach space and denote by $\K(\X)\subset \L(\X)$
the closed ideal of compact operators. Consider $\mathcal{C}(\X)$ the Calkin algebra over $X$,
i.e.,  the quotient algebra $\mathcal{C}(\X)=\L(\X)\K(\X)$. Recall that $\mathcal{C}(X)$ is itself a 
Banach algebra with the quotient norm. Let
$$
\pi\colon \L(\X)\to \mathcal{C}(\X)
$$
denote the quotient map.\par 

\indent  Recall that  $T\in \L(\X)$ is said to be a
\it semi-Fredholm operator, \rm if $R(T)$ is closed and $\alpha (T)= $ dim $N(T)$ or $\beta (T)=$ dim $X/R(T)$ is
finite, while if $\alpha (T)$ and $\beta (T)$ are both finite, $T$ is said to be \it Fredholm. \rm   Denote by $\Phi (\X)$ the set of all Fredholm operators defined on $\X$. It
is well known that $\Phi (\X)$ is a multiplicative open semigroup in $\L(\X)$ and that
$$
\Phi (\X)= \pi^{-1}(\mathcal{C}(\X)^{-1}).
$$

\indent  Given $T\in \L(\X)$, $T$ will be said to be \it power compact\rm, if
there exists $n\in \n$ such that $T^n\in \K(\X)$, equivalently, if $\pi (T)\in \mathcal{C}(X)$
is nilpotent. The set of all power compact operators defined on $\X$ will be denoted by $\mathscr{PK}(\X)$. In addition, $T\in \L(\X)$ will be said to be a \it Riesz operator\rm,
if $\pi (T)\in \mathcal{C}(\X)$ is quasi-nilpotent, equivalently 
$\sigma_e (T)=\sigma (\pi(T))=\{0\}$, where $\sigma_e (T)$ denotes the Fredholm spectrum of $T$ (see \cite{W1, W2} and \cite[Chapter 3]{CPY}).
Let $\mathscr{R}(\X)$ denote the set of Riesz operators defined on the Banach
space $\X$. It is clear that $\mathscr{PK}(\X)\subseteq \mathscr{R}(\X)$.\par

\indent Recall that according to the proof of \cite[Theorem 1]{R}, 
$$
\pi^{-1}(\mathcal{C}(\X)^\bullet)=\L(\X)^\bullet + \K(\X).
$$

\indent Atkinson's theorem motivated the development of an abstract Fredholm theory in 
arbitrary complex unital Banach algebras with respect to an \it inessential ideal, \rm i.e., a two sided
ideal $\J\subseteq \A$ such that for every $z\in \J$, $\sigma (z)$ is either finite or is a sequence
converging to zero. In fact,  $x\in \A$ is said to be a \it $\J$-Fredholm element \rm (or a \it Fredholm element of $\A$ relative to $\J$\rm),
if there exists $y\in \A$ such that $xy-1\in\J$ and $yx-1\in \J$. The set of all $\J$-Fredholm
elements will be denoted by $\Phi_{\J}(\A)$. In addition, given $\K\subset A$ a closed proper two sided ideal of $\A$,
$x\in \A$ will be said to be a \it Riesz element of $\A$ relative to  $\K$, \rm if $\pi (x)\in \A/\K$ is quasi-nilpotent,
where $\pi\colon \A\to \A/\K$ is the quotient map. The set of Riesz elements of $\A$ relative to $\K$
will be denoted $\mathscr{R}_{\K}(\A)$. To learn the main results concerning the Fredholm and the Riesz theory in
Banach algebras see \cite[Chapter 5]{A}, \cite{BMSW} and the corresponding bibliography of these monographs.
Recall that, according to \cite[Lemma 1]{Ba},  when the inessential ideal $\J$ is closed, 
$$
\pi^{-1}((\A/\J)^\bullet) =\A^\bullet+\J.
$$

\indent Atkinson's theorem also motivated the Fredholm theory relative to homomorphisms between two Banach algebras.  
In fact, let $\A$ and $\B$ be two unital Banach algebras and consider
$\T\colon \A\to \B$ an algebra homomorphism, not necessarily continue. Then, $\Phi_{\T}(\A)=\{a\in \A\colon \T(a)\in \B^{-1}\}$ is
the set of  Fredholm elements relative to the homomorphism $\T$,
see \cite{H1}. Moreover, Browder, Weyl and other classes of elements relative to $\T$ have been also intoduced and studied,
see for example \cite{H1, H3, H2, H4, MR1, GR, MR, D, H5, KMSD, Bak, ZH, MMH, ZDH1, ZDH2}.
To prove the main results of this article, it is important to recall another class of elements.
Let $\mathscr{R}_{\T}(A )=\{a\in \A\colon \T(a)\in \B^{qnil}\}$ be
the set of Riesz elements relative to the homomorphism $\T$, see \cite{ ZDH1, ZDH2}.
A subset of this class is the set of $\T$-nilpotent elements, i.e., $\mathscr{N}_{\T}(\A )=\{a\in \A\colon \hbox{ there exists } k\in\mathbb{N}
\hbox{ such that }a^k\in N(\T)\}$. Clearly, $\mathscr{N}_{\T}(\A )\subseteq \mathscr{R}_{\T}(A )$. \par

\indent Furthermore,   \cite[Theorem 1]{R}  was generalized to Banach algebras. In fact, 
consider $\A$ and $\mathcal{B}$  two complex unital Banach algebras and
$\mathcal{T}\colon \A\to \mathcal{B}$ an algebra homomorphism.
It will be said that  $\mathcal{T}$ has the \it Riesz property\rm, if $N(\mathcal{T})$
is an inessential ideal.
Then, according to \cite[Lemma 2]{D}, if $\mathcal{T}\colon \A\to \mathcal{B}$
has the Riesz property,
$$
\mathcal{T}^{-1}(B^\bullet)= A^\bullet + N(\mathcal{T}).
$$

Note that when $\A$ is a Banach algebra and $\J\subset\A$ is a closed  inessential ideal, then the
quotient map $\pi\colon \A\to \A/\J$ is a surjective Banach algebra homomorphism that has the Riesz property.
Moreover, $\Phi_{\pi}(\A)=\Phi_{\J}(\A)$ and $\mathscr{R}_{\pi}(\A)=\mathscr{R}_{\K}(A )$.
What is more, $\mathscr{N}_{\pi}(\A )=\{a\in \A\colon \hbox{ there exists } k\in\mathbb{N}\hbox{ such that }a^k\in \J\}$.
On the other hand, if $\A$ and $\B$ are two Banach algebras and $\T\colon \A\to\B$ is a surjective Banach
algebra homomorphism that has the Riesz property, then $\J=N(\T)$ is a closed inessential ideal of $\A$ and it is not difficult to prove that
$\Phi_{\T}(\A)=\Phi_{\J}(\A)$ and $\mathscr{R}_{\T}(\A)=\mathscr{R}_{\J}(\A)$.

\indent To fully characterize Drazin and Koliha-Drazin invertible Calkin algebra elements, some classes of operators
need to be considered.\par

\indent Recall that an operator $T\in\L (\X)$ is said to be \it semi-regular, \rm if $R(T)$ is closed and
$N(T^n)\subseteq R(T)$, for all $n\in\mathbb N$, see for example \cite{M1,M2, Mu, MM}. In addition, in \cite{Kat} it was proved that
given a semi-Fredholm operator $T\in \L(\X)$, there exist $M$ and $N$ two closed  subspaces of $\X$ invariant for $T$ such that
$\X= M\oplus N$,  $T\mid_{M}$ is nilpotent and $T\mid_{N}$ is semi-regular. This decomposition
is known as the \it Kato decomposition, \rm and the  operators satifying these conditions, which were characterized in \cite{La}, are said to be \it the quasi-Fredholm operators, \rm
see \cite{La, MM, Mu2}.\par

\indent On the othe hand, an operator $T\in \L (\X)$ is said to be \it B-Fredholm, \rm if there is $n\in \mathbb N$
such that $R(T^n)$ is closed and $T\mid_{R(T^n)}\in \L (R(T^n))$ is Fredholm, see \cite{Ber1, Ber2}.
In addition, according to \cite[Proposition 2.6]{Ber1}, a B-Fredholm operator is quasi-Fredholm;
what is more, according to \cite[Theorem 2.7]{Ber1}, if $T\in \L (\X)$ is B-Fredholm, then there exist
$M$ and $N$ two closed  subspaces of $\X$ invariant for $T$ such that
$\X= M\oplus N$,  $T\mid_{M}$ is nilpotent and $T\mid_{N}$ is Fredholm (see also \cite[Theorem 7]{Mu2}).  
The class of B-Fredholm operators defined on the Banach space $\X$ will be denoted by $\mathscr{BF}(\X)$.\par

\indent Quasi-Fredholm operators were generalized to pseudo-Fredholm operators.
In fact, given $T\in \L(\X)$, $T$ will be said to be a \it pseudo-Fredholm operator, \rm
if there exist $M$ and $N$ two closed spaces of $\X$ invariant for $T$ such that
$\X= M\oplus N$, $T\mid_{M}$ is quasi-nilpotent and $T\mid_{N}$ is semi-regular. 
This decomposition is called the \it generalized Kato decomposition, \rm see \cite{M0, M, M3, M4, Bou}.
In addition, an operator $T\in \L (\X)$ will be said to be \it pseudo B-Fredholm, \rm if there are two closed
subspaces of $\X$ invariant for $T$ such that 
$\X= M\oplus N$, $T\mid_{M}$ is quasi-nilpotent and $T\mid_{N}$ is is Fredholm. Let $\mathscr{PBF}(X)$
denote the class of pseudo B-Fredholm operators defined on $\X$. It is clear that
a pseudo B-Fredholm operator is pseudo-Fredholm. 

\indent Recall that according to \cite[Theorem 4]{K}, if $T\in \L(\X)$ is Drazin invertible, then
there exist two closed subspaces of $X$ invariant for $T$ such that $\X= M\oplus N$, $T\mid_{M}$ is nilpotent and $T\mid_{N}$ is invertible.
In particular, a Drazin invertible operator is B-Fredholm and B-Fredholm operators not only generalize Fredholm operators, but also Drazin invertible 
operators. In the next section it will be proved that the set of all cosets of B-Fredholm operators in a Calkin
algebra on an arbitrary Hilbert space coincides with the set of all Drazin invertible elements of the Calkin algebra.
Moreover,  the set of all Koliha-Drazin invertible elements of the same type of Calkin algebra will
be described using the class of pseudo B-Fredholm operator.
 
\section {\sfstp Main results}\setcounter{df}{0}

\indent To prove the main results of this work, first some facts need to be recalled.\par

\begin{rema}\label{rem41}\rm Let $\A$ be a unital Banach algebra and consider $a\in \A$. 
If  $0\neq p=p^2\in \A$, then define $p'=1-p$. \par
\noindent (i) Note that $a=pap+pap'+p'ap+p'ap'$.\par
\noindent (ii) Suppose that $ap=pa$, equivalently, $pa(1-p)=(1-p)ap=0$. 
Then, it is not difficult to prove that necessary and sufficient for $a$ to be invertible  is that
$pap\in (p\A p)^{-1}$ and $p'ap'\in (p'\A p')^{-1}$.\par
\noindent (iii) Let $\B$ another unital Banach algebra and consider $\T\colon \A\to \B$ an algebra homomorphism.
Suppose that $p\notin N(\T)$. Then, if $q=T(p)$ and $\T_{p,q}=\T\mid_{p\A p}^{q\B q}\colon p\A p\to q\B q$, $0\neq q=q^2$ and  $\T_{p,q}$ is an
algebra homomorphism between the unital Banach algebras $p\A p$ and  $q\B q$. What is more, if $\T$ is surjective, then $\T_{p,q}$
is surjective.
\end{rema}

\indent Next the main result will be presented. Note that the notation introduced in Remark \ref{rem41} will be used in this section.\par

\begin{thm}\label{thm42}Let $\A$ and $\B$ be two unital Banach algebras and consider an 
algebra homomorphism $\T\colon\A\to B$. Let $\lambda\in \mathbb{C}$ and $a\in \A$ and $b\in\B$ such that
$\T (a)=b$. Then, the following statements hold.\par
\noindent \rm (i) \it Suppose that  there exists an idempotent
$p\in \A$, $p\notin N(\T)$, such that $p(a-\lambda)(1-p)$ and $(1-p)(a-\lambda)p\in N(\T)$,
$(1-p)(a-\lambda)(1-p)\in\Phi_{\T_{p',q'}}(p'\A p')$ 
and $p(a-\lambda)p\in \mathscr{R}_{\T}(\A )$ \rm(\it respectively $p(a-\lambda)p\in \mathscr{N}_{\T}(\A )$, $p(a-\lambda)p\in \mathscr{R}_{\T}(\A )\setminus \mathscr{N}_{\T}(\A )$\rm)\it.
Then, $\lambda\in $ \rm iso \it $\sigma (b)$ \rm(\it respectively $\lambda\in \Pi(b)$, $\lambda\in I(b)$\rm)\it.  In particular, $0\neq q=\T(p)$ is the spectral idempotent of $b$
corresponding to $\lambda$.   \par
\noindent \rm (ii) \it Let $\lambda\in $ \rm iso \it $\sigma (b)$ \rm(\it respectively $\lambda\in \Pi(b)$, $\lambda\in I(b)$\rm)\it and let $0\neq q\in B$ the spectral  idempotent
of $b$ corrsponding to $\lambda$. If there exists $p\in A^\bullet$ such that $\T(p)=q$, then $p\notin N(\T)$,
$p(a-\lambda)(1-p)$ and $(1-p)(a-\lambda)p\in N(\T)$, 
$(1-p)(a-\lambda)(1-p)\in\Phi_{\T_{p',q'}}(p'\A p')$ 
and $p(a-\lambda)p\in \mathscr{R}_{\T}(\A )$ \rm(\it respectively $p(a-\lambda)p\in \mathscr{N}_{\T}(\A )$, $p(a-\lambda)p\in \mathscr{R}_{\T}(\A )\setminus \mathscr{N}_{\T}(\A )$\rm)\it.
\end{thm}
\begin{proof}(i). Suppose that  there exists an idempotent $p\in \A$, $p\notin N(\T)$, such that $p(a-\lambda)(1-p)$ and $(1-p)(a-\lambda)p\in N(\T)$, $p(a-\lambda)p\in \mathscr{R}_{\T}(\A )$ and $(1-p)(a-\lambda)(1-p)\in\Phi_{\T_{p',q'}}(p'\A p')$. Clearly, $q=\T (p)$ is and idempotent such that $0\neq q$ and $qb=bq$.
In addition, since $(a-\lambda)p= p(a-\lambda)p + (1-p)(a-\lambda)p$, $(1-p)(a-\lambda)p\in N(\T)$ and $p(a-\lambda)p\in \mathscr{R}_{\T}(\A )$, 
$(a-\lambda)p\in  \mathscr{R}_{\T}(\A )$, which in turn implies that $(b-\lambda)q
\in \B^{qnil}$. Moreover, it is clear that $(1-q)(b-\lambda)(1-q)\in (q'\B q')^{-1}$
($\T_{p',q'}\colon p'\A p'\to q'\B q'$ is as in \rm Remark \ref{rem41}(iii) ($p'=1-p$ and $q'=1-q$)). Now well, note that 
$(b-\lambda +q)= q(b-\lambda)q +q + (1-q)(b-\lambda)(1-q)$ and 
$q(b-\lambda +q)q=q(b-\lambda)q +q =(b-\lambda)q +q$. However, since $(b-\lambda)q\in \B^{qnil}$, 
$((b-\lambda)q +q)\in (q\B q)^{-1}$. Therefore, according to Remark \ref{rem41}(ii), $(b-\lambda +q)\in \B^{-1}$,
which, according to Remark \ref{rem11}(i), implies that $\lambda\in $ iso $\sigma (b)$.\par

\indent If $p(a-\lambda)p\in \mathscr{N}_{\T}(\A )$, then 
proceed as in the previous paragraph and note that since
$p(a-\lambda)p\in \mathscr{N}_{\T}(\A)$, then $q(b-\lambda)q=(b-\lambda )q\in \B$ is nilpotent. 
The remaining part of the implication
follows as for the case $p(a-\lambda)p\in \mathscr{R}_{\T}(\A )$  using in particular Remark \ref{rem11}(ii) instead of Remark \ref{rem11}(i)

\indent The case $p(a-\lambda)p\in \mathscr{R}_{\T}(\A )\setminus \mathscr{N}_{\T}(\A )$
is  a direct consequence of what has been proved.\par

\noindent  (ii). Suppose that $\lambda\in $ \rm iso $\sigma (b)$ and consider $0\neq q=q^2\in \B$ the spectral idempotnt of $b$ corresponding to $\lambda$.
In particular, $qb=bq$, $(b-\lambda)q$
is quasi-nilpotent and $(b-\lambda +q)\in \B^{-1}$. Suppose that there exists $p=p^2\in \A$ such that $\T (p)=q$.
As a result, $p\notin N(\T)$. Decompose $a\in A$, $\T (a)=b$, as follows: $a=pap + pa(1-p) +(1-p)ap+(1-p)a(1-p)$.
Note that since $qb=bq$ and $\T$ is an algebra homomorphism, $p(a-\lambda)(1-p)$ and $(1-p)(a-\lambda)p\in N(\T)$.
Moreover, since $\T (p(a-\lambda)p)=q(b-\lambda)q=(b-\lambda)q\in\B^{qnil}$, $p(a-\lambda)p\in\mathscr{R}_{\T}(\A)$.
Next consider $\T_{p',q'}\colon p'\A p'\to q'\B q'$.
Since $(b-\lambda +q)\in \B^{-1}$, it is not difficult to prove that $(1-q)(b-\lambda)(1-q)\in (q'\B q')^{-1}$ (Remark \ref{rem41}(ii)). Then, clearly, $(1-p)(a-\lambda)(1-p)\in\Phi_{\T_{p',q'}}(p' \A p')$.\par

\indent Suppose that $\lambda\in \Pi(b)$ and proceed as in the previous paragraph. 
Note that since $(q(b-\lambda)q)^k=((b-\lambda)q)^k=0$, $k=$ ind $(b-\lambda)$, 
$((a-\lambda)p)^k\in N(\T)$. In particular, $(a-\lambda)p  \in\mathscr{N}_{\T}(\A)$. The remaining part of the implication
follows as in the previous paragraph.\par
\indent  As in the proof of statement (i), the case $\lambda\in I(b)$
is  a direct consequence of what has been proved.
\end{proof}

\indent Next Drazin and Koliha-Drazin Banach algebra invertible elements in homomorphism ranges will be characterized.\par

\begin{cor}\label{cor43}Let $\A$ and $\B$ be two unital Banach algebras and consider an algebra
homomorphism $\T\colon\A\to B$. Let $a\in \A$ and $b\in\B$ such that
$\T (a)=b$. Then, the following statements hold.\par
\noindent \rm (i) \it Suppose that  there exists an idempotent
$p\in \A$, $p\notin N(\T)$, such that $pa(1-p)$ and $(1-p)ap\in N(\T)$,
$(1-p)a(1-p)\in\Phi_{\T_{p',q'}}(p'\A p')$ 
and $pap\in \mathscr{R}_{\T}(\A )$ \rm(\it respectively $pap\in \mathscr{N}_{\T}(\A )$, $pap\in \mathscr{R}_{\T}(\A )\setminus \mathscr{N}_{\T}(\A )$\rm)\it.
Then, $b\in\B$ is Koliha-Drazin invertible \rm but not invertible (\it respectively $b$ is Drazin invertible but not invertible, $b$ is Koliha-Drazin invertible but not Drazin invertible\rm)\it. \par
\noindent \rm (ii) \it Suppose that $b$ is Koliha-Drazin invertible but not invertible \rm(\it respectively Drazin invertible but not invertible, Koliha-Drazin invertible but not Drazin invertible)
and let $q\in B$ the spectral idempotent
of $b$ corresponding to $0$. If there exists $p\in A^\bullet$ such that $T(p)=q$, then $p\notin N(\T)$,
$pa(1-p)$ and $(1-p)ap\in N(\T)$, 
$(1-p)a(1-p)\in\Phi_{\T_{p',q'}}(p'\A p' )$ 
and $pap\in \mathscr{R}_{\T}(\A )$ \rm(\it respectively $pap\in \mathscr{N}_{\T}(\A )$, $pap\in \mathscr{R}_{\T}(\A )\setminus \mathscr{N}_{\T}(\A )$\rm)\it.
\end{cor}
\begin{proof} The Corollary can be easily deduced from Theorem \ref{thm42} taking in consideration the following 
observations. Recall that $b\in\B$ is Koliha-Drazin invertible but not invertible (respectively Drazin invertible but not invertible)
if and only if $0\in$ iso $\sigma(b)$ (\cite[Theorem 4.2]{Ko2}) (respectively $0\in \Pi(b)$ (\cite[Proposition 1.5]{Lu} or \cite[Theorem 12(i)]{Bo})).
\end{proof}

Note that in Corollary \ref{cor43} (ii), if $b$ is group invertible, then according to Remark \ref{rem11}(ii), 
$pap\in N(\T)$.

To fully characterize poles, isolated points and Drazin and Koliha-Drazin invertible elements in quotient Banach algebras
or homomorphism ranges, it is necessary to be able to lift spectral idempotents. In the following remark some examples for 
which the statements in Theorem \ref{thm42} and Corollary \ref{cor43}
are equivalent will be presented.\par

\begin{rema}\label{rem44}\rm (i). Let $\A$ be a unital Banach algebra and consider $\mathscr{R}$ the radical of $\A$ and
$\pi\colon \A\to\A/\mathscr{R}$ the quotient map.
Then, according to \cite[Theorem 2.3.9]{Ri}, any idempotent in $\A/\mathscr{R}$ can be lifted to an idempotent in $\A$.\par
\noindent (ii) Let $\A$ be a unital Banach algebra and consider $\J\subset\A$ a closed inessential ideal. Let $\pi\colon \A\to \A/\J$. Then, according to 
 \cite[Lemma 1]{Ba},  $\pi^{-1}((\A/\J)^\bullet) =\A^\bullet+\J$. In the case of a $C^*$-algebra, see also \cite[Theorem 15]{Bar}.
Note that according to \cite[Theorem 2.3.5]{Ri}, the radical of a Banach algebra $\A$ is a closed  inessential ideal.\par

\noindent (iii). Let $\A$ and $\B$ be two unital Banach algebras and consider $\T\colon \A\to \B$ a surjective  algebra homomorphism
such that $\T$ has the Riesz property. Then, according to  \cite[Lemma 2]{D}, $\mathcal{T}^{-1}(B^\bullet)= A^\bullet + N(\mathcal{T})$.\par

\noindent (iv) Let $\A$ be a $C^*$-algebra and consider $\mathscr{I}\subseteq \A$ a closed two-sided ideal. Let $\pi\colon \A\to\A/\mathscr{I}$
be the quotient map. Given $b\in \A/\mathscr{I}$ such that $b^2=b$, according to \cite[Corollary 3]{Had} and \cite[Corollary 4]{Had},
if $\A$ has real rank zero or is a von Neumann algebra, then there exists $a\in \A$ such that $\pi (a)=b$ and $a^2=a$.  
The same result was proved when $\A$ is a $C^*$-algebra, $\pi\colon \A\to\A/\mathscr{I}$ is the quotient map and 
the closed two sided ideal $\mathscr{I}$ satisfies Condition (A) in \cite[p.  24]{Choi} (for example if 
$\mathscr{I}$ is an $\hbox{\rm AW}^*$ algebra (\cite{Kap}) or an AF $C^*$-algebra  (\cite{Brat})), see \cite[Theorem 1]{Choi}. \par

\noindent (v) Note that if $\A$ is a Banach algebra and $\K\subset \A$ is a closed two sided ideal, 
if $b^2-b\in \K$, in general there is no $a\in \A$, $a^2=a$ such that $a-b\in \K$. In fact, consider $\A=C[0,1]$ and
$\K=\{f\in C[0.1]\colon f(0)=f(1)=0\}$ (see \cite[p. 431]{Had}).\par
\end{rema}

\indent In the following theorem conditions that fully characterized Drazin and Koliha-Drazin invertible elements in quotient Banach algebras
or in Banach algebras that can be presented as the range of a (not necessarily continuous) homomorphism will be given. To this end, 
the following notion will be first introduced. Let $\A$ and $\B$ two complex unital Banach algebras
and consider $\T\colon \A\to\B$ a surjective algebra homomorphism. $\T$ will be said to have the \it lifting property, \rm if given
$q\in \B^\bullet$, there is $p\in\A^\bullet$ such that $\T(p)=q$, equivalently, $\T^{-1}(\B^\bullet)=\A^\bullet + N(\T)$.
Note that the homomorphism considered in Remark \ref{rem44}(i)-(iv) have the lifting property.

\begin{thm}\label{thm92}Let $\A$ and $\B$ be two unital Banach algebras and consider a surjective 
algebra homomorphism $\T\colon\A\to B$ such that $\T$ has the lifting property. Let $a\in \A$ and $b\in\B$ such that
$\T (a)=b$. Then, $b$ is Koliha-Drazin invertible but not invertible (respectively Drazin invertible but not invertible,
Koliha-Drazin invertible but not Drazin invertible) if and only if there exist $p^2=p\in \A$, $p\notin N(\T)$, $x\in \A$, $pxp\in\mathscr{R}_{\T}(\A)$   (respectively 
$pxp\in \mathscr{N}_{\T}(\A )$,  $pxp\in  \mathscr{R}_{\T}(\A )\setminus \mathscr{N}_{\T}(\A )$), $y\in (1-p)\A(1-p)$,   $y\in\Phi_{\T_{p',q'}}(p'\A p')$ ($q=\T (p)$) and 
$z\in N(\T)$ such that $a=pxp+y+z$.
\end{thm}
\begin{proof}Apply Remark \ref{rem11}, Remark \ref{rem41} and Theorem \ref{thm42} or Corollary \ref{cor43}.
\end{proof}

\begin{rema}\label{rem97}\rm Under the same hypothesis and notatition of Theorem \ref{thm92}, note that $b\in \B$ is group invertible if and only if $pap\in N(T)$ (Remark \ref{rem11}(ii)), so that,
in this case, it is possible to chose $x=0$.  Moreover, when Theorem \ref{thm92} holds, it is possible to consider $a=x$, $y=(1-p)a(1-p)$ and $z=pa(1-p) + (1-p)ap$ (Theorem \ref{thm42} or Corollary \ref{cor43}). \par
\noindent In addition, given $\lambda\in \mathbb{C}$, to characterize when $\lambda\in$ iso $\sigma (b)$
(respectively $\lambda\in \Pi (b)$, $\lambda\in$ iso $\sigma (b)\setminus\Pi (b)$), it is enough to consider when $b-\lambda$ is Koliha-Drazin invertible
but not invertible (respectively Drazin invertible but not invertible, Koliha-Drazin invertible but not Drazin invertible), see the proof of Corolary \ref{cor43}.
\end{rema}
Naturally, the main application of Theorem \ref{thm42} and Corollary \ref{cor43} is the Calkin algebra.

\begin{thm}\label{thm21}Let $\X$ be a complex Banach space and consider $T\in \L(\X)$, 
$\lambda\in\mathbb{C}$ and $\pi (T)\in \C(\X)$. The following statements hold.\par
\noindent \rm (i) \it Necessary and sufficient for $\lambda\in $ \rm iso \it $\sigma (\pi (T))$ is that there exists an idempotent
$P\in \L(\X)$, $P\notin\K (\X)$, such that if $\X_1=P(\X)$ and $\X_2=(I-P) (\X)$, then there are $T_1\in \mathscr{R}(\X_1)$,
$T_2\in \Phi (\X_2)$ and $K\in \K(\X)$ such that $T-\lambda=T_1\oplus T_2+K$.\par
\noindent \rm (ii) \it  The number $\lambda\in \mathbb{C}$ is a pole of $\pi (T)$ if and only if there exists an idempotent
$P\in \L(\X)$, $\X_1$, $\X_2$, $T_2$ and $K$ as in statement \rm (i) \it  and $T_1\in \mathscr{PK}(\X_1)$,
such that $T-\lambda=T_1\oplus T_2+K$.\par
\noindent \rm (iii) \it   A necessary and sufficient condition for $\lambda\in I(\pi (T))$ is that  there exists an idempotent
$P\in \L(\X)$, $\X_1$, $\X_2$, $T_2$ and $K$ as in statement \rm (i) \it  and $T_1\in  \mathscr{R}(\X_1)\setminus\mathscr{PK}(\X_1)$,
such that $T-\lambda=T_1\oplus T_2+K$. \par
\end{thm}
\begin{proof}Adapt the proof ot Theorem \ref{thm42} to the case under consideration,
using in particular \cite[Theorem 1]{R}.
\end{proof}

\indent Under the same hypothesis and notation in Theorem \ref{thm21}, to characterize when $\pi (T)\in \C(\X)$
is Koliha-Drazin invertible but not invertible (respectively Drazin invertible but not invertible, Koliha-Drazin invertible but not Drazin invertible),
consider the case $\lambda=0$  (see the proof of Corolary \ref{cor43}). In addition,
 when $\pi (T)\in \C(\X)$ is group invertible, according again to Remark \ref{rem11}(ii), note that
$T_1\in \K(\X)$, so that, in this case, $T=0\oplus T_2 +K$.

\indent Next the Hilbert space case will be consider. Compare with \cite[Theorem 2.2]{Bel}.
Note that according to  \cite[Section 4]{Kon}, the structure theorem for polynomially compact operators
proved in \cite[Theorem 2.4]{O} holds for arbitrary  Hilbert spaces.\par

\begin{cor}\label{cor23}Let $\H$ be a  Hilbert space and consider $T\in \L(\H)$ and $\pi (T)\in \C(\H)$. The following
statments hold.\par
\noindent\rm (i) \it  $\lambda\in$ iso $\sigma (\pi (T))$ if and only if 
there exists an idempotent $P\in \L(\H)$, $P\notin\K (\H)$, such that if $\H_1=P(\H)$ and $\H_2=(I-P) (H)$, then  there are a quasi-nilpotent operator $T_1\in \L(\H_1)$,
$T_2\in \Phi (\H_2)$ and $K\in \K(\H)$ such that $T-\lambda=T_1\oplus T_2+K$. 

\noindent \rm (ii) \it $\lambda\in\Pi (\pi (T))$  if and only if  there exists 
$P\in \L(\H)$, $\H_1$, $\H_2$ and $T_2\in \L(\H_2)$ as in statement \rm (i)  \it and a nilpotent operator $T_1\in \L(\H_1)$ such that $T-\lambda=T_1\oplus T_2+K$. 

\noindent\rm (iii) \it  $\lambda\in I(\pi (T))$ if and only if 
there exists $P\in \L(\H)$,  $\H_1$ and $\H_2$, $T_2\in \L(\H_2)$ as in statement \rm (i)  \it and a quasi-nilpotent but not nilpotent operator $T_1\in \L(\H_1)$ such that $T-\lambda=T_1\oplus T_2+K$.
\end{cor}
\begin{proof}

\noindent (i). According to Theorem \ref{thm21}(i), it is enough to prove the necessary condition. To this end, according to Theorem \ref{thm21} (i),
there exist  an idempotent
$P\in \L(\H)$, $\H_1=R(P)$, $\H_2=R(I-P)$,  $T'_1\in \mathscr{R}(\H_1)$, $T_2\in\Phi (\H_2)$ and $K'\in \K(\H)$
such that $T-\lambda=T'_1\oplus T_2+K'$. However, according to the West decomposition of Riesz operators
(\cite[Theorem 7.5]{W2}), there is $K_1\in \K(\H_1)$ such that $T'_1+ K_1$ is quasi-nilpotent.  To conclude de proof, Define $T_1=T'_1+K_1$ and 
$$
K=K'+ \begin{pmatrix}
-K_1&0\\
0&0\\
\end{pmatrix}.
$$
\noindent (ii) Proceed as in the proof of statement (i) but instead of \cite[Theorem 7.5]{W2}
use the  structure theorem of polynomially compact operators of C. Olsen (\cite[Theorem 2.4]{O} and  \cite[Section 4]{Kon}).

\noindent (iii). It is a consequence of statements (i) and (ii).
\end{proof}

\indent In the same conditions of Corollary \ref{cor23}, to characterize when $\pi (T)\in \C(\H)$ is Koliha-Drazin invertible but not invertibe
(respectively Drazin invertible but not invertible, Koliha-Drazin invertible but not Drazin invertible), it is enough to consider
$\lambda=0$ and statement (i) (respectively statement (ii), (iii)), see  Corollary \ref{cor43}.\par

\indent Corollary \ref{cor23} can be formulated in $C^*$-algebras.\par

\begin{pro}\label{prop35}Let $\A$ be a $C^*$-algebra and consider $\J\subset\A$ a closed  inessential 
ideal. Let $a\in \A$ and $b\in \A/\J$ such that $b=\pi (a)$ ($\pi\colon \A\to \A/\J$ the quotient map).
Then $b$ is Koliha-Drazin invertible but not invertible (respectively Drazin invertible but not invertible,
Koliha-Drazin invertible but not Drazin invertible) if and only if there exist $p^2=p\in \A$, $p\notin \J$, $x\in \A $ quasi-nilpotent (respectively 
nilpotent, quasi-nilpotent but not nilpotent) $y\in (1-p)\A(1-p)$, $y\in\Phi_{\pi_{p',q'}}(p'\A p')$ ($q=\pi (p)$), and $z\in \J$ such that  $a=pxp+y+z$.
\end{pro}
\begin{proof}Suppose that there exist $p^2=p\in \A$, $p\notin \J$, $x\in \A^{qnil}$, $y\in (1-p)\A(1-p)$, $y\in\Phi_{\pi_{p',q'}}(p'\A p')$ and $z\in \J$ such that 
 $a=pxp+y+z$. Then, according to \cite[Theorem 1.6.15]{Ri}, $pxp\in (p\A p)^{qnil}\subseteq \mathscr{R}_{\J}(\A)$.
Therefore, acording to Theorem \ref{thm92}, $b\in \A/\J$ is Koliha-Drazin invertible.\par

\indent On the other hand, if $b\in \A/\J$ is Koliha-Drazin invertible, then according to Theorem \ref{thm92},
there exist $p^2=p\in \A$, $p\notin \J$, $x'\in \A$, $px'p\in \mathscr{R}_\J(\A)$, $y\in (1-p)\A(1-p)$, $y\in\Phi_{\pi_{p',q'}}(p'\A p')$ and $z'\in \J$ such that 
$a=px'p+y+z'$. However, according to \cite[Corollary $C^*$.2.5]{BMSW}, there are $s$ quasi-nilpotent and $t\in \J$
such that $px'p=s+t$ (the West decompostion in $C^*$-algebras). As a result, $px'p= psp+ptp$. Let $x=s$ and $z=z'+ptp$. Then,
$a=pxp+y+z$, $x\in \A^{qnil}$,  $y\in (1-p)\A(1-p)$, $y\in\Phi_{\pi_{p',q'}}(p'\A p')$ and  $z\in \J$.\par
\indent To prove the equivalence for the case $b\in \A/\J$ Drazin invertible but not invertible, proceed as before using \cite[Theorem 15]{Bar}
instead of \cite[Corollary $C^*$.2.5]{BMSW}.\par
\indent The case $b\in \A/\J$ Koliha-Drazin invertible but not Drazin invertible is a consequence of what has been proved.
\end{proof} 

\begin{rema}\label{rema1000}\rm Under the same hypothesis of Proposition \ref{prop35}, the case $\lambda\in $ iso $\sigma(b)$ (respectively $\lambda\in\Pi (b)$,
$\lambda\in I(b)$) can be characterized as in Remark \ref{rem97}. Note also that, according to \cite[Theorem 1.6.15]{Ri}, if $x\in \A^{qnil}$, then 
$pxp\in (p\A p)^{qnil}$.  \par
\noindent On the other hand,  a statement similar to the one in Proposition \ref{prop35} can be
consider for the case $\T\colon \A\to \B$, $\A$ and $\B$ two unital Banach algebras and $\T$ a surjective Banach algebra homomorphism
that has the Riesz property (see section 2). 
\end{rema}

\indent Finally, Drazin and Koliha-Drazin invertible elements in Calkin algebras will be characterized using classes of operators.
Recall that an Atkinson-type theorem for B-Fredholm operators holds. 
In fact, according to \cite[Theorem 3.4]{BS}, $T\in \L (\X)$ is a B-Fredholm operator if and only its coset in $\L (\X)/F_0 (\X)$
is Drazin invertible, where $F_0 (\X)$ is the ideal of finite rank operators in $\L (\X)$. What is more, in $\C (\X)$ this characterization
does not hold (see \cite[Remark p. 256]{BS}). However, in the case of a Calkin algebra on a Hilbert space, Drazin invertible elements
will be characterized using B-Fredholm operators. To consider the Banach space case, two classes of operators need to be introduced.\par

\indent Let $\X$ be a Banach space and consider $T\in \L (\X)$. The operator $T$ will be said to be \it Riesz-Fredholm \rm 
(respectively \it power compact-Fredholm \rm), if there exist $M$ and $N$ two closed subspaces of $\X$ invariant for $T$
such that $\X=M\oplus N$, $T\mid_{M}\in \mathscr{R}(M)$ (respectively $T\mid_{M}\in \mathscr{PK}(M)$)
and $T\mid_{N}\in \Phi (N)$. These classes of operators will be denoted by $\mathscr{RF}(\X)$ and $\mathscr{PKF}(\X)$
respectively.

\begin{thm}\label{thm120} Let $\X$ be a Banach space and consider $\pi\colon \L (\X)\to\C (\X)$. The following statements hold.\par
\noindent \rm (i) \it $\pi (\mathscr{RF}(\X))=\C(\X)^{KD}$, equivalently, $\pi^{-1}(\C(\X)^{KD})=\mathscr{RF}(\X) + \K(\X)$.\par
\noindent \rm (ii) \it $\pi (\mathscr{PKF}(\X))=\C(\X)^{D}$, equivalently, $\pi^{-1}(\C(\X)^{D})=\mathscr{PKF}(\X) + \K(\X)$.
\end{thm}
\begin{proof} Apply Theorem \ref{thm21}.
\end{proof}

\begin{thm}\label{thm140} Let $\H$ be a Banach space and consider $\pi\colon \L (\H)\to\C (\H)$. The following statements hold.\par
\noindent \rm (i) \it $\pi (\mathscr{PBF}(\H))=\C(\H)^{KD}$, equivalently, $\pi^{-1}(\C(\H)^{KD})=\mathscr{PBF}(\H) + \K(\H)$.\par
\noindent \rm (ii) \it $\pi (\mathscr{BF}(\H))=\C(\H)^{D}$, equivalently, $\pi^{-1}(\C(\H)^{D})=\mathscr{BF}(\H) + \K(\H)$.
\end{thm}
\begin{proof} Apply Corollary \ref{cor23}.
\end{proof}

\vskip.5truecm
\noindent Enrico Boasso\par
\noindent E-mail: enrico\_odisseo@yahoo.it

\end{document}